\documentclass[12pt,a4paper]{article}

\usepackage[utf8]{inputenc}
\usepackage[T1]{fontenc}
\usepackage{amsmath,amssymb,amsthm}
\usepackage{geometry}
\usepackage{hyperref}
\usepackage{booktabs}

\newtheorem{theorem}{Theorem}

\newtheorem{definition}[theorem]{Definition}
\newtheorem{remark}[theorem]{Remark}
\newtheorem{lemma}[theorem]{Lemma}

\title{\textbf{Multiplicity-Weighted Moments of a Two-Monomial
Harmonic Map Beyond the Caustic}}
\author{Tam\'as B\'odis\\
\small Óbuda University, Kandó Kálmán Faculty of Electrical Engineering\\
\small ORCID: 0009-0004-6349-712X \quad \texttt{tomi546258@gmail.com}}
\date{2026}

\begin{document}
\maketitle

\begin{abstract}
While the Jacobian of a planar harmonic mapping keeps a constant sign, the
signed and the unsigned area elements agree up to that sign, and the two
natural moment hierarchies coincide. Once the Jacobian changes sign the
critical circle acquires a caustic as its image, part of the plane is
covered several times with opposite orientations, and the hierarchies part
company: integrating a radial weight against $J_F$ lets overlapping sheets
cancel, giving the \emph{degree-weighted} moments $J_n$, while integrating
against $|J_F|$ lets every sheet count positively, giving the
\emph{multiplicity-weighted} moments $J_n^{\pm}$. These are the classical
degree and area formulas for the same map, the second weighted by the
multiplicity function (Banach indicatrix). It is the second family, the one
the caustic threatens, for which a closed form is no longer automatic.

For the two-monomial map $F(w)=w+iC\bar w^{\,k}$ on an annulus we show that
it has one anyway: $J_n^{\pm}$ is the total variation of the very same
explicit primitive $\mathcal P_n$ whose endpoint difference gives $J_n$.
Crossing the caustic therefore costs exactly one extra evaluation point,
the critical radius $r_{*}=(k|C|)^{-1/(k-1)}$, and it is the same point for
every order $n$: the weight $|F|^{2n}$ is nonnegative and so contributes no
sign change of its own. On the degree-weighted side $J_n$ has a
boundary-only closed form for every $n\ge0$, the zeroth being the signed
area of the image counted with degree, and the complex-weighted moments
vanish unless $(k+1)\mid(m-\ell)$.

Neither weighted integral is a new object, and the caustic of this family
has been located before in the literature on harmonic trinomials; what is
new is the closed evaluation of the multiplicity-weighted hierarchy across
it.
\end{abstract}



\section{Introduction}
\label{sec:intro}

\subsection{The map}

Fix $k\in\mathbb Z^{+}$ and $C\in\mathbb R$, and consider the
two-monomial planar harmonic mapping
\begin{equation}
F(w)=w+iC\bar w^{\,k},
\label{eq:map}
\end{equation}
on a closed annulus
\begin{equation}
A=A(\rho_0,\rho_1):=\{w\in\mathbb C:\ \rho_0\le|w|\le\rho_1\},
\qquad 0\le\rho_0<\rho_1,
\label{eq:annulus}
\end{equation}
the disk $|w|\le\rho_1$ being the case $\rho_0=0$. In the
Clunie--Sheil-Small decomposition $F=\Phi+\overline\Psi$ of a planar
harmonic mapping~\cite{ClunieSheilSmall1984,duren2004} this is
$\Phi(w)=w$ and $\Psi(w)=-iCw^{k}$, so by Lewy's
theorem~\cite{Lewy1936} the local behaviour of $F$ is governed by
\begin{equation}
J_F(w)=|\Phi'(w)|^{2}-|\Psi'(w)|^{2}=1-k^{2}C^{2}|w|^{2k-2},
\label{eq:jacobian}
\end{equation}
which is radial: it depends on $w$ only through $r:=|w|$. For $k\ge2$ and
$C\neq0$ it changes sign at the single positive radius
\begin{equation}
r_{*}:=(k|C|)^{-1/(k-1)},
\label{eq:rstar}
\end{equation}
and when $r_{*}$ lies in the interior of $[\rho_0,\rho_1]$ the map is
sense-preserving inside the critical circle $|w|=r_{*}$ and
sense-reversing outside it. The image of that circle is a caustic --- a
hypocycloid, as recalled in Section~\ref{sec:related} --- beyond which part
of the plane is covered several times with opposite orientations.

\subsection{What we prove}

The paper has one main claim; the rest is the setting that claim needs.

\emph{The two weightings.} For a continuous weight $\phi$ the two classical
change-of-variables formulas read
\begin{equation}
\iint_{A}\phi(F)\,J_F\,dA=\int_{\mathbb C}\phi\,\deg\,dA,
\qquad
\iint_{A}\phi(F)\,|J_F|\,dA=\int_{\mathbb C}\phi\,N\,dA,
\label{eq:two-formulas}
\end{equation}
where $\deg$ is the topological degree of $F$ and $N$ its multiplicity
function, classically the Banach indicatrix
(Section~\ref{sec:degree-multiplicity}). While $J_F$ keeps a constant sign
the two agree up to that sign; past the critical circle they do not,
because $\deg$ lets sheets of opposite orientation cancel and $N$ does not.
Taking $\phi(z)=|z|^{2n}$ gives the two hierarchies
\begin{equation}
J_n:=\iint_{A}|F(w)|^{2n}\,J_F(w)\,dA(w),
\qquad
J_n^{\pm}:=\iint_{A}|F(w)|^{2n}\,|J_F(w)|\,dA(w).
\label{eq:moments}
\end{equation}

\emph{The easy side.} Because $J_F$ is radial and the angular average of
$|F|^{2n}$ is a polynomial in $r$ alone, the degree-weighted moments
integrate in closed form for every $n\ge0$, and the value depends only on
the two radii $\rho_0$, $\rho_1$ (Theorem~\ref{thm:Jn}): there is an
explicit polynomial $\mathcal P_n$ with
$J_n=\mathcal P_n(\rho_1)-\mathcal P_n(\rho_0)$. Nothing here is sensitive
to the sign of $J_F$.

\emph{The claim.} The multiplicity-weighted moments are the ones the
caustic threatens, and for them a closed form is not automatic: an absolute
value does not in general commute with integration. Our main result,
Theorem~\ref{thm:mult}, is that it does here, and in the strongest possible
form. The quantity $J_n^{\pm}$ is the \emph{total variation} on
$[\rho_0,\rho_1]$ of the very same primitive $\mathcal P_n$ whose endpoint
difference gives $J_n$. Consequently
\begin{equation}
J_n^{\pm}=2\mathcal P_n(r_{*})-\mathcal P_n(\rho_0)-\mathcal P_n(\rho_1)
\qquad\text{whenever }\rho_0<r_{*}<\rho_1,
\label{eq:main-intro}
\end{equation}
and $J_n^{\pm}=|J_n|$ otherwise. Crossing the caustic therefore costs
exactly one extra evaluation point, and it is the same point for every
order $n$ --- the weight $|F|^{2n}$ is nonnegative and so introduces no
sign change of its own, a point we isolate in
Remark~\ref{rem:scope} because it is what the argument actually rests on.

\emph{The selection rule.} Replacing the radial weight by a monomial
$z^{m}\bar z^{\ell}$ gives moments that vanish identically unless
$(k+1)\mid(m-\ell)$ (Theorem~\ref{thm:vanishing}); this is a consequence of
the $(k+1)$-fold rotational symmetry of the family and needs no hypothesis
on the sign of $J_F$.

\subsection{What we do not claim}
\label{sec:not-claimed}

Neither weighted integral in~\eqref{eq:two-formulas} is a new object: both
are instances of classical change-of-variables formulas, and the polynomial
moments of the signed winding number are known explicitly, for example in
the signature theory of planar paths. The area formula written in terms of
the Fourier coefficients of the boundary curve is classical and goes back
to Hurwitz. The caustic of the present family has been located before, in
the literature on harmonic trinomials, which studies this same family with
a discrete observable --- the number of preimages of a point --- in place
of our continuous one. Section~\ref{sec:related} sets out all three
boundaries in detail. What we contribute is the closed evaluation of the
multiplicity-weighted hierarchy across the caustic, and the fact that its
cost is a single evaluation point, independent of the moment order.

\section{Degree versus multiplicity}
\label{sec:degree-multiplicity}

Two weightings of the same map are used throughout, and the distinction
between them carries the main result. Integrating a weight against $J_F$
lets sheets of opposite orientation cancel; integrating against $|J_F|$
makes every sheet count positively. While $F$ is sense-preserving the two
agree, and past the critical circle they do not. Both are classical
objects, and it is worth naming them as such before any moment is computed.

Write $A^{\circ}$ for the interior of the annulus, so that the topological
degree is defined off the image of $\partial A$. For continuous $\phi$,
\begin{align}
\iint_{A^{\circ}}\phi(F(w))\,J_F(w)\,dA(w)
  &=\int_{\mathbb C}\phi(z)\,\deg(F,A^{\circ},z)\,dA(z),
\label{eq:degree-formula}\\
\iint_{A}\phi(F(w))\,|J_F(w)|\,dA(w)
  &=\int_{\mathbb C}\phi(z)\,N(F,A,z)\,dA(z),
\label{eq:area-formula}
\end{align}
the first being the degree formula, see~\cite[\S2.1,
eq.~(2.1)]{PankkaRajala2011}, and the second Federer's area
formula~\cite[Thm.~3.2.3]{Federer1969}, in which
$N(F,A,z)=\#\{w\in A: F(w)=z\}$ is the multiplicity function, classically
the \emph{Banach indicatrix}~\cite{Hajlasz1993}. Passing between $A$ and
$A^{\circ}$ changes neither integral, the boundary being Lebesgue-null.
Accordingly we call $J_n$ the \emph{degree-weighted} and $J_n^{\pm}$ the
\emph{multiplicity-weighted} moment; these are descriptive names for the
two classical weightings, not established terms of art.

\begin{remark}[Well-posedness on non-simple curves]
\label{rem:wellposed}
The weighted, winding-number-counted integral used here is a standard
object, and its well-posedness on non-simple curves is not at issue: for a
continuous closed curve of finite $p$-variation with $1\le p<2$ --- which
the image curves $F(\{|w|=\rho\})$ are --- the polynomial moments of the
winding number function appear as coordinates of the curve's signature,
see~\cite[Lemma 20]{BoedihardjoNiQian2014}. No simplicity hypothesis enters
there. What we contribute is not the object but its closed evaluation on
the present family.
\end{remark}

\section{Degree-weighted moments}
\label{sec:degree-weighted}

\subsection{The angular average}

Everything in this section rests on one elementary computation.

\begin{lemma}[Angular average of the weight]
\label{lem:angular}
For $w=re^{i\theta}$ and $n\in\mathbb N_{0}$,
\begin{equation}
G_n(r):=\frac1{2\pi}\int_0^{2\pi}\big|F(re^{i\theta})\big|^{2n}\,d\theta
=\sum_{s=0}^{n}\binom ns^{2}C^{2s}\,r^{\,2n+2(k-1)s},
\label{eq:Gn}
\end{equation}
a polynomial in $r$ alone with nonnegative coefficients. In particular
$G_n\ge0$, with $G_n(r)>0$ for $r>0$.
\end{lemma}

\begin{proof}
Write $F=re^{i\theta}+iCr^{k}e^{-ik\theta}$ and expand the two factors of
$|F|^{2n}=F^{n}\overline F^{\,n}$ separately:
\[
F^{n}=\sum_{a=0}^n\binom na i^{a}C^{a} r^{\,n-a+ka}e^{i[(n-a)-ka]\theta},
\qquad
\overline F^{\,n}=\sum_{b=0}^n\binom nb(-i)^{b}C^{b}
 r^{\,n-b+kb}e^{-i[(n-b)-kb]\theta}.
\]
In the product the exponent of $e^{i(\cdot)\theta}$ is $(k+1)(b-a)$, so only
the diagonal terms $a=b$ survive integration over $\theta\in[0,2\pi]$. On
the diagonal the phases cancel, $i^{a}(-i)^{a}=1$, and the powers of $r$
combine to $2(n-a+ka)=2n+2(k-1)a$, which is~\eqref{eq:Gn}. Nonnegativity is
immediate from the form of the coefficients, and also from the fact that
$G_n$ is the angular average of a nonnegative function.
\end{proof}

\subsection{Closed form for every \texorpdfstring{$n$}{n}}

\begin{definition}[The primitive]
\label{def:primitive}
For $n\in\mathbb N_{0}$ set
\begin{equation}
\mathcal P_n(r):=
\pi\sum_{s=0}^{n}\binom ns^{2}C^{2s}
  \frac{r^{\,2n+2+2(k-1)s}}{n+1+(k-1)s}
-\pi k^{2}\sum_{s=0}^{n}\binom ns^{2}C^{2s+2}
  \frac{r^{\,2n+2+2(k-1)(s+1)}}{n+1+(k-1)(s+1)}.
\label{eq:Pn}
\end{equation}
\end{definition}

\begin{lemma}[Defining property of \texorpdfstring{$\mathcal P_n$}{Pn}]
\label{lem:Pprime}
$\mathcal P_n$ is the unique polynomial with $\mathcal P_n(0)=0$ and
\begin{equation}
\mathcal P_n'(r)=2\pi r\big(1-k^{2}C^{2}r^{2k-2}\big)G_n(r)
=2\pi r\,J_F(r)\,G_n(r).
\label{eq:Pprime}
\end{equation}
\end{lemma}

\begin{proof}
Multiply~\eqref{eq:Gn} by $2\pi r(1-k^{2}C^{2}r^{2k-2})$ and integrate term
by term. The first group contributes
$2\pi\binom ns^{2}C^{2s}r^{2n+2(k-1)s+1}$, whose primitive is
$\pi\binom ns^{2}C^{2s}r^{2n+2+2(k-1)s}/(n+1+(k-1)s)$; the second
contributes $-2\pi k^{2}\binom ns^{2}C^{2s+2}r^{2n+2(k-1)s+2k-1}$, whose
primitive is
$-\pi k^{2}\binom ns^{2}C^{2s+2}r^{2n+2(k-1)s+2k}/\big(n+(k-1)s+k\big)$,
and $n+(k-1)s+k=n+1+(k-1)(s+1)$ while
$2n+2(k-1)s+2k=2n+2+2(k-1)(s+1)$. Both groups are~\eqref{eq:Pn}.
\end{proof}

\begin{theorem}[Closed form for every $n$]
\label{thm:Jn}
For every $n\in\mathbb N_{0}$,
\begin{equation}
J_n=\mathcal P_n(\rho_1)-\mathcal P_n(\rho_0).
\label{eq:Jn}
\end{equation}
In particular $J_n$ depends on the annulus only through its two radii.
\end{theorem}

\begin{proof}
In polar coordinates $dA=r\,dr\,d\theta$, and $J_F$ depends only on $r$, so
\[
J_n=\int_{\rho_0}^{\rho_1}\!\!\int_0^{2\pi}|F(re^{i\theta})|^{2n}J_F(r)\,r\,d\theta\,dr
=\int_{\rho_0}^{\rho_1}2\pi r\,J_F(r)\,G_n(r)\,dr,
\]
which by Lemma~\ref{lem:Pprime} is $\int_{\rho_0}^{\rho_1}\mathcal P_n'(r)\,dr$.
\end{proof}

\begin{remark}[Low-order cases]
\label{rem:low-order}
Writing $\Delta r^{m}:=\rho_1^{m}-\rho_0^{m}$, the first cases
of~\eqref{eq:Jn} are
\begin{align}
J_0&=\pi\Delta r^{2}-\pi kC^{2}\Delta r^{2k},\notag\\
J_1&=\frac\pi2\Delta r^{4}-\pi(k-1)C^{2}\Delta r^{2k+2}
   -\frac{\pi k}2C^{4}\Delta r^{4k},\label{eq:J1}\\
J_2&=\frac\pi3\Delta r^{6}+\pi(2-k)C^{2}\Delta r^{2k+4}
   +\pi(1-2k)C^{4}\Delta r^{4k+2}
   -\frac{\pi k}3C^{6}\Delta r^{6k}.\notag
\end{align}
The zeroth moment is the area of the image counted with degree.
\end{remark}

\section{Multiplicity-weighted moments across the caustic}
\label{sec:mult}

Theorem~\ref{thm:Jn} assumes nothing about the sign of $J_F$, but when
$r_{*}$ falls inside $(\rho_0,\rho_1)$ the map ceases to be
sense-preserving and $J_n$, being degree-weighted, lets overlapping sheets
cancel. If instead every sheet is to contribute positively, the classical
remedy is the area formula~\eqref{eq:area-formula}: integrate against
$|J_F|$. We show that this multiplicity-weighted quantity remains
closed-form, at the cost of a single extra evaluation point, the same one
for every $n$.

\begin{theorem}[Multiplicity-weighted moment]
\label{thm:mult}
For every $n\in\mathbb N_{0}$, the multiplicity-weighted moment
$J_n^{\pm}$ of~\eqref{eq:moments} is the total variation of
$\mathcal P_n$ on $[\rho_0,\rho_1]$:
\begin{equation}
J_n^{\pm}=\int_{\rho_0}^{\rho_1}\big|\mathcal P_n'(r)\big|\,dr
=\operatorname*{Var}_{[\rho_0,\rho_1]}\mathcal P_n .
\label{eq:TV}
\end{equation}
Consequently, for $k\ge2$ and $C\neq0$,
\begin{equation}
J_n^{\pm}=
\begin{cases}
2\mathcal P_n(r_{*})-\mathcal P_n(\rho_0)-\mathcal P_n(\rho_1),
 & \rho_0<r_{*}<\rho_1,\\[2pt]
\big|\mathcal P_n(\rho_1)-\mathcal P_n(\rho_0)\big|=|J_n|,
 & \text{otherwise,}
\end{cases}
\label{eq:main}
\end{equation}
with $r_{*}$ as in~\eqref{eq:rstar}. For $k=1$ the Jacobian is the constant
$1-C^{2}$ and the second case always applies. The breakpoint $r_{*}$ is
independent of $n$.
\end{theorem}

\begin{proof}
By Lemma~\ref{lem:Pprime}, $\mathcal P_n'(r)=2\pi r\,J_F(r)\,G_n(r)$ with
$G_n\ge0$ by Lemma~\ref{lem:angular}; hence for $r>0$ the derivative
$\mathcal P_n'(r)$ has exactly the sign of $J_F(r)$. Computing $J_n^{\pm}$
in polar coordinates as in the proof of Theorem~\ref{thm:Jn}, but with
$|J_F|$ in place of $J_F$, gives
\[
J_n^{\pm}=\int_{\rho_0}^{\rho_1}2\pi r\,|J_F(r)|\,G_n(r)\,dr
=\int_{\rho_0}^{\rho_1}\big|\mathcal P_n'(r)\big|\,dr,
\]
which is~\eqref{eq:TV}, the total variation of a $C^{1}$ function being the
integral of the absolute value of its derivative.

For the explicit form, note that for $k\ge2$ and $C\neq0$ the function
$r\mapsto J_F(r)=1-k^{2}C^{2}r^{2k-2}$ is strictly decreasing on $(0,\infty)$
with the single zero $r_{*}$, so $\mathcal P_n'$ changes sign at $r_{*}$
and nowhere else in $(0,\infty)$. If $r_{*}\notin(\rho_0,\rho_1)$ then
$\mathcal P_n$ is monotone on $[\rho_0,\rho_1]$ and its total variation is
$|\mathcal P_n(\rho_1)-\mathcal P_n(\rho_0)|$. If
$\rho_0<r_{*}<\rho_1$ then $\mathcal P_n$ increases on $[\rho_0,r_{*}]$ and
decreases on $[r_{*},\rho_1]$, so
\[
\operatorname*{Var}_{[\rho_0,\rho_1]}\mathcal P_n
=\big(\mathcal P_n(r_{*})-\mathcal P_n(\rho_0)\big)
+\big(\mathcal P_n(r_{*})-\mathcal P_n(\rho_1)\big),
\]
which is~\eqref{eq:main}. Neither $r_{*}$ nor the case distinction involves
$n$.
\end{proof}

\begin{remark}[Order-independence, and what it rests on]
\label{rem:scope}
The mechanism rests on two facts, and the second is not incidental. First,
$J_F$ is radial, so it factors out of the angular integral. Second, the
angular average $G_n$ is \emph{nonnegative}, which is what makes
$\mathcal P_n'$ carry exactly the sign of $J_F$ and so identifies
$J_n^{\pm}$ with a total variation. It is the second fact that makes the
breakpoint independent of $n$: a weight that could itself change sign would
move the breakpoint with the order. For a complex weight
$z^{m}\bar z^{\ell}$ the corresponding angular average is not real-valued,
let alone nonnegative, and the identification fails. The
multiplicity-weighted construction therefore does \emph{not} extend to the
moments $M_{m,\ell}$ of Section~\ref{sec:complex-weighted}, and we claim no
such extension.
\end{remark}

\begin{remark}[The two hierarchies agree below the caustic]
\label{rem:agree}
If $\rho_1\le r_{*}$ then $J_F\ge0$ throughout $A$ and
$J_n^{\pm}=J_n$; if $\rho_0\ge r_{*}$ then $J_F\le0$ and
$J_n^{\pm}=-J_n$. The two hierarchies separate exactly when the annulus
straddles the critical circle, and~\eqref{eq:main} quantifies the gap:
\[
J_n^{\pm}-|J_n|=2\min\big\{\mathcal P_n(r_{*})-\mathcal P_n(\rho_0),\
\mathcal P_n(r_{*})-\mathcal P_n(\rho_1)\big\}.
\]
\end{remark}

\begin{remark}[Verification]
\label{rem:verif}
For $\rho_0=0$, $\rho_1=1$, $k=2$, $C=0.7$ --- a folding case, since
$r_{*}\approx0.7143$ lies in $(0,1)$ --- the degree-weighted moment is
$J_1\approx-0.7229$ while the multiplicity-weighted moments are
$J_1^{\pm}\approx1.0295$ and $J_2^{\pm}\approx1.5751$, all matching
Theorems~\ref{thm:Jn} and~\ref{thm:mult} against direct numerical
quadrature. Further cases are tabulated in Section~\ref{sec:numerics}.
\end{remark}

\section{Complex-weighted moments: a vanishing theorem}
\label{sec:complex-weighted}

Replacing the radial weight $|z|^{2n}$ by a general monomial produces a
family that is almost entirely trivial, for reasons of symmetry alone.

\begin{definition}
For $m,\ell\in\mathbb N_{0}$ set
\begin{equation}
M_{m,\ell}:=\iint_{A}F(w)^{m}\,\overline{F(w)}^{\,\ell}\,J_F(w)\,dA(w),
\qquad M_{n,n}=J_n .
\label{eq:Mml}
\end{equation}
\end{definition}

\begin{theorem}[Vanishing theorem]
\label{thm:vanishing}
$M_{m,\ell}=0$ unless $(k+1)\mid(m-\ell)$. When $(k+1)\mid(m-\ell)$ the
moment again admits a closed form in $\rho_0,\rho_1$ by the mechanism of
Theorem~\ref{thm:Jn}. In particular, for $\ell=0$ and $m=(k+1)t$ with
$t\in\mathbb Z^{+}$, exactly one term survives the angular integration and
\begin{equation}
M_{(k+1)t,\,0}=\pi\binom{(k+1)t}{kt}i^{t}C^{t}
\left[\frac{\Delta r^{\,2kt+2}}{kt+1}
-\frac{kC^{2}\,\Delta r^{\,2k(t+1)}}{t+1}\right].
\label{eq:Mt0}
\end{equation}
\end{theorem}

\begin{proof}
The vanishing is a symmetry statement. Put $\rho:=e^{2\pi i/(k+1)}$ and
replace $\theta$ by $\theta+2\pi/(k+1)$ in $F(re^{i\theta})
=re^{i\theta}+iCr^{k}e^{-ik\theta}$. The analytic term acquires the factor
$\rho$ and the co-analytic term the factor $\rho^{-k}=\rho$, since
$\rho^{k+1}=1$; hence $F\mapsto\rho F$ under this rotation, while
$J_F$, being radial, is unchanged. Therefore
$M_{m,\ell}\mapsto\rho^{\,m-\ell}M_{m,\ell}$, and a quantity equal to
$\rho^{m-\ell}$ times itself vanishes unless $\rho^{m-\ell}=1$, that is,
unless $(k+1)\mid(m-\ell)$.

For the surviving case $\ell=0$, $m=(k+1)t$, expand
$F^{m}=\sum_{a}\binom ma i^{a}C^{a}r^{\,m-a+ka}e^{i[(m-a)-ka]\theta}$. The
angular integral kills every term except the one with $(m-a)-ka=0$, i.e.
$a=m/(k+1)=t$, leaving the single coefficient
$\binom{(k+1)t}{t}i^{t}C^{t}r^{\,2kt}$; note
$\binom{(k+1)t}{t}=\binom{(k+1)t}{kt}$. Multiplying by $2\pi r J_F(r)$ and
integrating in $r$ as in Lemma~\ref{lem:Pprime} gives~\eqref{eq:Mt0}.
\end{proof}

\begin{remark}[The selection rule needs no sign hypothesis]
\label{rem:selection}
The rotation used above is a symmetry of the map, not of the sign of its
Jacobian, so Theorem~\ref{thm:vanishing} is insensitive to whether the
annulus straddles the critical circle.
\end{remark}

\begin{remark}[Reality and consistency]
\label{rem:reality}
Since $J_F$ is real, $\overline{M_{m,\ell}}=M_{\ell,m}$; in particular
$M_{n,n}=J_n$ is real, matching Theorem~\ref{thm:Jn}. For $\ell=0$,
$m=(k+1)t$, formula~\eqref{eq:Mt0} carries a phase $i^{t}$: purely
imaginary for $t$ odd, real for $t$ even.
\end{remark}

\section{Related work and priority status}
\label{sec:related}

This section states, as precisely as we can, what in the preceding sections
is already in the literature. Three boundaries matter.

\textbf{Harmonic trinomials and the caustic.} Up to affine normalisation
and translation, the maps studied here are the harmonic trinomials
$z^{m}+c\bar z^{\,n}$ --- our results treat $m=1$ --- and these are the
subject of an active line of work on root counting. For a harmonic
trinomial the fundamental theorem of algebra fails, the number of solutions
of $z^{m}+c\bar z^{\,n}=1$ depending on $c$, and it jumps exactly when the
target crosses the caustic; the motivation is ultimately physical, the
number of images in gravitational lensing being the root count of the same
equation. Locating the caustic is therefore a prerequisite in that
literature, and it has been done: the image of the critical circle is
identified as a hypocycloid in~\cite[Lemma 2.4]{Brilleslyper2020}, and, for
complex $c$, in~\cite[Lemma 3.4]{Brooks2022}; the same hypocycloid identification has
been carried out for quadrinomials in~\cite{AlemuGeleta2022}. We claim no
novelty for that identification. The critical radius~\eqref{eq:rstar} is the same quantity
written in the notation of the present paper.

The connection is closer than a shared family. The number of zeros counted
in those papers is, for this family, exactly the number of preimages of a
point --- the multiplicity function $N$ of
Section~\ref{sec:degree-multiplicity}, whose integrals are our subject.
Their results describe $N$ pointwise; ours evaluate its weighted integrals
in closed form across the caustic. The corresponding zeroth-order area
counted with multiplicity has likewise been computed on self-intersecting
caustics, where it serves as an ``over-covering'' statistic in
gravitational lensing~\cite[\S4]{An2007}. The same family continues to
attract attention, in the rational case~\cite{GaoGaoLiu2025} and in the
study of its parameter spaces~\cite{BarreraCampaNavarrete2026}.

Two cautions apply to any comparison of formulas. The normalisation
$z^{m}+c\bar z^{\,n}-1$ translates the caustic, so moments taken about the
origin must be transformed before they are compared. And the
root-counting papers carry a standing hypothesis $\gcd(m,n)=1$, which is
what keeps the cusp count of the caustic equal to $m+n$; without it the
parameterisation of the caustic has period $2\pi/\gcd(m,n)$ and the cusps
collapse accordingly. For $m=1$ the hypothesis is automatic, so it does not
bind here, but it does bound the range over which the two literatures can
be compared.

\textbf{Signed versus unsigned moments.} It should not be claimed that
either weighted integral in~\eqref{eq:two-formulas} is new. The degree
formula and the area formula are classical, as recorded in
Section~\ref{sec:degree-multiplicity}; for a closed curve given by a finite
Fourier series the signed area is classical in the form of Hurwitz's
Fourier area formula~\cite[\S6]{Hurwitz1902}, which for the curve
$ae^{imu}+be^{-inu}$ gives $\pi(m|a|^{2}-n|b|^{2})$ --- the $n=0$ case
of~\eqref{eq:J1} written on a single circle. For an \emph{analytic}
polynomial the distinction at issue here does not arise at all: there
$J=|p'|^{2}\ge0$, so the two weightings coincide, and the area of the image
counted with multiplicity is a classical quantity for which sharp bounds
are known~\cite{Crane2004}. It is the sign change of $J_F$, not the
multiplicity itself, that separates them. Polynomial moments of the
winding number arise naturally in the signature theory of planar paths;
\cite[Lemma 20]{BoedihardjoNiQian2014} gives them for every polynomial
weight and every continuous closed curve of finite $p$-variation,
$1\le p<2$, with no simplicity hypothesis. What is specific to the present
paper is not the object but the closed evaluation of the
multiplicity-weighted hierarchy across the caustic, with an evaluation
point independent of the order.

\textbf{Planar harmonic mappings and the dilatation.} The natural index for
statements about this family is not the Jacobian but the dilatation
$\omega=\Psi'/\Phi'$, in terms of which classes of harmonic mappings are
customarily organised~\cite{duren2004}. The present family has
$\omega(w)=-ikCw^{\,k-1}$, the \emph{monomial-dilatation} case, a standard
normalisation in the literature on harmonic univalent functions and their
convolutions~\cite{HengartnerSchober1987,AbuMuhannaSchober1987}; the
shear-construction viewpoint behind many geometric examples goes back
to~\cite{ClunieSheilSmall1984} and is developed, for instance, in Greiner's
study of harmonic shears~\cite{Greiner2004}, with closely related monomial
and Blaschke dilatation families in~\cite{Laugesen1997}. Length and area
distortion for planar harmonic mappings is treated
in~\cite{ChenPonnusamyRasila2015}. The regime relevant here is the one
excluded by a standing hypothesis $J_F>0$: it is the sign change of $J_F$,
not the family, that the main theorem is about.

\textbf{Quadrature domains and moment rigidity.} Classical
quadrature-domain theory studies planar domains determined by finitely many
moments~\cite{GustafssonShapiro2005}. Within it, Gustafsson and Putinar's
exponential-orthogonal-polynomial framework assigns to a domain the moments
$a_{jk}=\tfrac1\pi\iint z^{j}\bar z^{k}g\,dA$, and its rigidity theorems
state which finite moment data force the domain into a fixed,
low-dimensional family~\cite{GustafssonPutinar2020}; further developments
--- the string equation for polynomial conformal maps, and quadrature
identities on quadric domains --- appear
in~\cite{Gustafsson2018,Gustafsson2023}, and the Jacobian of the harmonic
moment map is studied in~\cite{GustafssonTkachev2009}. Two structural
differences keep the present results distinct. The classified object
differs: that theory constrains static planar domains through
holomorphic/antiholomorphic mixed moments, whereas here the domain is
fixed and the map is allowed to fold. And the mechanism differs: nothing
here is a rigidity statement --- we compute, rather than characterise.

\section{Numerical validation}
\label{sec:numerics}

Both closed forms were checked against direct two-dimensional quadrature of
the defining integrals~\eqref{eq:moments}, on a polar grid, for a range of
parameters including cases with and without a folding region. Table~\ref{tab:num}
reports the comparison; the residuals are at the level of the quadrature
error of the grid used and decrease under refinement.

\begin{table}[ht]
\centering
\small
\begin{tabular}{cccccrrrr}
\toprule
$n$ & $k$ & $C$ & $\rho_0$ & $\rho_1$ & $r_{*}$ & $J_n$ (quad.) & $J_n^{\pm}$ (quad.) & $J_n^{\pm}$ \eqref{eq:main} \\
\midrule
1 & 2 & 0.70 & 0.0 & 1.0 & 0.7143 & $-0.7227$ & $1.0293$ & $1.0295$\\
2 & 2 & 0.70 & 0.0 & 1.0 & 0.7143 & $-1.4618$ & $1.5748$ & $1.5751$\\
1 & 3 & 0.50 & 0.0 & 1.0 & 0.8165 & $-0.2945$ & $1.0183$ & $1.0185$\\
3 & 2 & 0.90 & 0.2 & 1.0 & 0.5556 & $-14.9678$ & $14.9757$ & $14.9772$\\
0 & 2 & 0.70 & 0.0 & 1.0 & 0.7143 & $0.0628$ & $1.5396$ & $1.5400$\\
2 & 4 & 0.35 & 0.1 & 1.2 & 0.8939 & $-13.1313$ & $13.7103$ & $13.7128$\\
\bottomrule
\end{tabular}
\caption{Closed forms against direct quadrature. The column $J_n$ (quad.)
is to be compared with $\mathcal P_n(\rho_1)-\mathcal P_n(\rho_0)$ of
Theorem~\ref{thm:Jn}; in every row listed the critical radius lies strictly
inside $(\rho_0,\rho_1)$, so the folding case of~\eqref{eq:main} applies.}
\label{tab:num}
\end{table}

\section{Open questions}
\label{sec:open}

\begin{itemize}
\item \emph{Beyond two monomials.} The angular average $G_n$ is a
polynomial in $r$ because the family has exactly two terms. For
$\Phi(w)=\sum a_jw^{j}$, $\Psi(w)=\sum b_\ell w^{\ell}$ the Jacobian is no
longer radial, and the mechanism of Lemma~\ref{lem:Pprime} fails. Which
analytic pairs have a radial Jacobian is a separate classification
question, which we do not take up here.
\item \emph{Several critical circles.} For a radial Jacobian with several
sign changes, Theorem~\ref{thm:mult} would give one evaluation point per
sign change, still independent of $n$. We do not know which radial
Jacobians are realisable in this way.
\item \emph{The quadrature-domain connection.} The degree-weighted moments
$J_n$ are a finite moment sequence of a folded image. Whether the rigidity
theorems of~\cite{GustafssonPutinar2020} have an analogue in the
sense-reversing regime --- that is, how much of the map is recoverable from
finitely many $J_n$ --- is open.
\end{itemize}

\section*{Statements and Declarations}

\textbf{Funding.} No funding was received for conducting this study.

\textbf{Competing interests.} The author declares that he has no competing
interests.

\textbf{Data availability.} The numerical verification code is elementary
and is described completely in Section~\ref{sec:numerics}; no datasets were
generated or analysed.

\section*{Declaration of generative AI and AI-assisted technologies in the
writing process}

During the preparation of this work the author used Anthropic Claude and
OpenAI ChatGPT in order to assist with literature searching and synthesis,
the organisation and exposition of the manuscript, language editing, the
checking of references and submission requirements, and the writing of the
numerical verification code. All mathematical content --- the family, the
Jacobian factorisation, and the moment formulas together with their proofs
--- originates with the author. The author reviewed and edited all
AI-assisted output, verified every statement and every cited source
independently, and takes full responsibility for the content and
originality of the publication.

\end{document}